\documentclass{amsart}
\usepackage{amsmath,amssymb}
\usepackage[dvips]{graphics}
\usepackage[dvipdfmx]{graphicx}
\usepackage[all]{xy}

\newtheorem{Theorem}{Theorem}[section]
\newtheorem{Lemma}[Theorem]{Lemma}
\newtheorem{Proposition}[Theorem]{Proposition}

\theoremstyle{definition}

\theoremstyle{remark}
\newtheorem{Remark}[Theorem]{Remark}

\def\labelenumi{\rm ({\makebox[0.8em][c]{\theenumi}})}
\def\theenumi{\arabic{enumi}}
\def\labelenumii{\rm {\makebox[0.8em][c]{(\theenumii)}}}
\def\theenumii{\roman{enumii}}

\renewcommand{\thefigure}
{\arabic{section}.\arabic{figure}}
\makeatletter
\@addtoreset{figure}{section}
\def\@thmcountersep{-}
\makeatother

\numberwithin{equation}{section}

\begin{document} 

\title[Intrinsic linking of simplicial $n$-complexes in $\mathbb{R}^{2n}$]{Intrinsic linking of simplicial $n$-complexes in $\mathbb{R}^{2n}$: An additional minimal $n$-complex}

\author{Ryo Nikkuni}
\address{Department of Information and Mathematical Sciences, School of Arts and Sciences, Tokyo Woman's Christian University, 2-6-1 Zempukuji, Suginami-ku, Tokyo 167-8585, Japan}
\email{nick@lab.twcu.ac.jp}
\thanks{The author was supported by JSPS KAKENHI Grant Number 25K06987.}


\date{}

\dedicatory{}


\begin{abstract}
For any positive integer $n$, the author previously constructed several minimal simplicial $n$-complexes which necessarily contain a non-splittable two-component link, consisting of an $(n-1)$-sphere and an $n$-sphere, in any embedding into $\mathbb{R}^{2n}$. In this paper, we present an additional simplicial $n$-complex with the same property.
\end{abstract}

\maketitle

\section{Introduction} 

Throughout this paper, we work in the piecewise linear category. For the fundamentals of piecewise-linear topology, we refer the reader to \cite{H69} and \cite{RS82}. Let $K$ be a simplicial $n$-complex. In this paper, a simplicial complex is always identified with its polyhedron. Let $m$ be a positive integer and let $p$ and $q$ be non-negative integers satisfying $m = p + q + 1$. Then $K$ may have the property that, no matter how it is embedded in the $m$-dimensional Euclidean space $\mathbb{R}^{m}$, it always contains a non-splittable two-component link consisting of a $p$-sphere and a $q$-sphere. In the case $m = 2n + 1$, every simplicial $n$-complex $K$ can be embedded in $\mathbb{R}^{2n+1}$. Furthermore, the case $p = q = n$ is especially important in higher-dimensional link theory, and many such simplicial $n$-complexes are known (Segal--Spie\.z \cite{SeSp92}, Lov\'{a}sz--Schrijver \cite{LS98}, Skopenkov \cite{Sk03}, Taniyama \cite{T00}, and Nikkuni \cite{N26}). In particular, when $n=1$, $K$ is said to be \textit{intrinsically linked}. Such simplicial $1$-complexes were first discovered by Conway--Gordon \cite{CG83} and Sachs \cite{S84}, and the complete characterization by Robertson--Seymour--Thomas \cite{RST95} is especially well-known.

On the other hand, when $m=2n$, note that $K$ cannot always be embedded in $\mathbb{R}^{2n}$. For $p=n-1$ and $q=n$, an example of such a complex was first constructed by Freedman--Krushkal \cite[Proposition 4.1]{FK14} (see also Freedman--Krushkal--Teichner \cite[Lemma 6]{FKT94}). In particular, when $n=1$, a simplicial complex $K$ possessing this property was characterized by Dehkordi--Farr \cite{DF}. In this context, a simplicial $1$-complex $K$ that does not possess the property is said to be \textit{non-separating planar}. It has been shown that $K$ is non-separating planar if and only if it does not contain a subdivision of any of the graphs $K_{1} \sqcup K_{4}$, $K_{1} \sqcup K_{2,3}$, or $K_{1,1,3}$, as shown in Fig.~\ref{nonsp}. As higher-dimensional generalizations of these simplicial $1$-complexes, for any positive integer $n$, the minimal simplicial $n$-complexes $N_{1}^{(n)}$, $N_{2}^{(n)}$, and $N_{3}^{(n)}$ possessing this property were constructed in the author's previous work \cite{N26b}. Independently, and particularly for $n=2$, Huber--Rao--Schwartz--Vijay showed that the suspension of $K_6$ yields a simplicial $2$-complex possessing this property \cite{HRJV26}.

\begin{figure}[htbp]
\begin{center}
\scalebox{0.475}{\includegraphics*{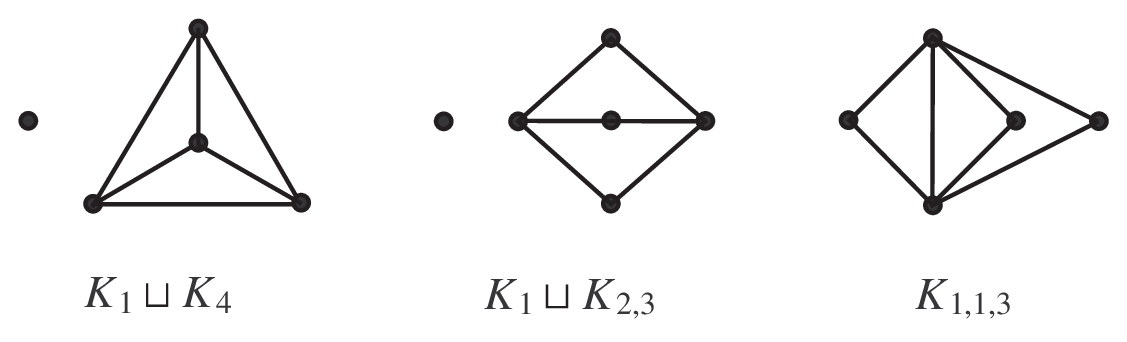}}
\caption{$K_1 \sqcup K_4$, $K_1 \sqcup K_{2,3}$ and $K_{1,1,3}$}
\label{nonsp}
\end{center}
\end{figure}

Our purpose in this paper is to refine and generalize the result of \cite{HRJV26} to higher dimensions by constructing a new minimal simplicial $n$-complex $M^{(n)}$ that possesses the property for every positive integer $n$. In the following, for a simplicial $n$-complex $K$, we denote the set of all $k$-simplices of $K$ by $\varDelta^{k}(K)$. Then $K^{k} = \bigcup_{l \le k} \varDelta^{l}(K)$ is the \textit{$k$-skeleton} of $K$, which is a subcomplex of $K$. Let $f \colon K \to \mathbb{R}^{m}$ be an embedding. Let $\Lambda^{p,q}(K)$ be the set of all unordered pairs $\{\gamma, \delta\}$ of mutually disjoint subcomplexes of $K$ such that $\gamma$ and $\delta$ are homeomorphic to the $p$-sphere and the $q$-sphere, respectively. We identify each unordered pair $\{\gamma,\delta\}$ with the disjoint union $\gamma \sqcup \delta$. For any $\lambda = \gamma \sqcup \delta \in \Lambda^{p,q}(K)$, the image $f(\lambda) = f(\gamma) \sqcup f(\delta)$ forms a two-component link in $f(K)$ consisting of a $p$-sphere and a $q$-sphere. If $p + q = m-1$, the \textit{$\mathbb{Z}_2$-linking number}
\begin{eqnarray*}
{\rm lk}_2(L) = {\rm lk}_2(S^p,S^q) = {\rm lk}_2(S^q,S^p) \in \mathbb{Z}_2
\end{eqnarray*}
of a two-component link $L = S^p \sqcup S^q$ in $\mathbb{R}^m$ is well-defined (cf.\ \cite{ST80}). Let $\sigma_{m} = |a_{0}a_{1}\cdots a_{m}|$ be an $m$-simplex whose vertices are $a_{0},a_{1},\ldots,a_{m}$. Note that $\sigma_{m}$ is itself a simplicial $m$-complex, and its $k$-skeleton is denoted by $\sigma_{m}^{k}$. In particular, when $m=5$ and $k=1$, $\sigma_5^1$ is isomorphic to $K_6$, a graph well-known to be intrinsically linked \cite{CG83}, \cite{S84}. Let $S(K_6)$ denote the suspension of $K_6$, that is, the join of $K_6$ with two additional vertices $a$ and $b$ not belonging to $K_6$. Then, Huber--Rao--Schwartz--Vijay showed the following.

\begin{Theorem}\label{ILR4} 
{\rm (Huber--Rao--Schwartz--Vijay \cite{HRJV26})} 
For every embedding $f \colon S(K_6) \to \mathbb{R}^{4}$, there exists an element $\lambda \in \Lambda^{1,2}(S(K_6))$ such that ${\rm lk}_2(f(\lambda)) = 1$. 
\end{Theorem}

In fact, their argument essentially shows that 
\begin{eqnarray*}
\sum_{\lambda \in \Lambda^{1,2}(S(K_6))} {\rm lk}_2(f(\lambda)) \equiv 2 \pmod{4},
\end{eqnarray*}
from which it follows that there exist at least two links with ${\rm lk}_2 = 1$.

In the following, we generalize Theorem \ref{ILR4} to higher dimensions. For any positive integer $n$, let $J$ be a simplicial $(n-1)$-complex, and let $a$, $b$, and $c$ be three vertices not belonging to $J$. Then, as a subcomplex of the simplicial $n$-complex $J * \{a,b,c\}$ (the join of $J$ and $\{a,b,c\}$), we define the simplicial $n$-complex $M_J^{(n)}$ by
\begin{eqnarray*}
M_{J}^{(n)} 
&=& \bigl( J * \{a,b,c\} \bigr) \setminus \{ c * s \mid s \in \varDelta^{n-1}(J) \} \\
&=& \bigl( J * \{a,b\} \bigr) \cup \bigl\{ c * t \mid t \in \varDelta^{n-2}(J) \bigr\}.
\end{eqnarray*}
In particular, when $J = \sigma_{2n}^{n-1}$, we denote $M_{\sigma_{2n}^{n-1}}^{(n)}$ by $M^{(n)}$. For $n \ge 2$, it is well known that $\sigma_{2n}^{n-1}$ is a minimal simplicial $(n-1)$-complex that cannot be embedded in $\mathbb{R}^{2n-2}$ \cite{VK33}, \cite{Flores32}. When $n=2$, $\sigma_4^1$ is isomorphic to $K_5$. In other words, $M^{(n)}$ is obtained by removing all $n$-simplices of the form $|c a_{i_0} \cdots a_{i_{n-1}}|$ ($0 \le i_0 < i_1 < \cdots < i_{n-1} \le 2n$) from $\sigma_{2n}^{n-1} * \{a,b,c\}$. Equivalently, it is obtained by adding all $(n-1)$-simplices of the form $|c a_{i_0} \cdots a_{i_{n-2}}|$ ($0 \le i_0 < i_1 < \cdots < i_{n-2} \le 2n$) to $\sigma_{2n}^{n-1} * \{a,b\}$. Note that $M^{(n)}$ can be regarded as a subcomplex of the suspension $S(\sigma_{2n+1}^{n-1})$ of $\sigma_{2n+1}^{n-1}$ by identifying $c$ with the additional vertex $a_{2n+1}$. Since $\sigma_{2n+1}^{n-1}$ can be embedded in $\mathbb{R}^{2n-1}$, its suspension $S(\sigma_{2n+1}^{n-1})$ can be embedded in $\mathbb{R}^{2n}$. Therefore, $M^{(n)}$ can also be embedded in $\mathbb{R}^{2n}$. Then we obtain the following.

\begin{Theorem}\label{ILR2n} 
Let $n$ be a positive integer. For every embedding $f \colon M^{(n)} \to \mathbb{R}^{2n}$, there exists an element $\lambda \in \Lambda^{n-1,n}(M^{(n)})$ such that ${\rm lk}_2(f(\lambda)) = 1$. 
\end{Theorem}

Note that $K_6$ is isomorphic to $\sigma_5^1$ and that $S(\sigma_5^1)$ contains $M^{(2)}$ as a proper subcomplex. In general, for any positive integer $n$, the suspension $S(\sigma_{2n+1}^{n-1})$ contains exactly $2n+2$ subcomplexes isomorphic to $M^{(n)}$, and each element of $\Lambda^{n-1,n}(S(\sigma_{2n+1}^{n-1}))$ is shared by exactly $n+1$ of these subcomplexes. Thus, for any embedding of $S(\sigma_{2n+1}^{n-1})$ into $\mathbb{R}^{2n}$, Theorem \ref{ILR2n} implies that there exist at least $(2n+2)/(n+1) = 2$ two-component links with ${\rm lk}_2 = 1$. We will prove Theorem \ref{ILR2n} in the next section. While the method in \cite{HRJV26} is based on the original method of Conway--Gordon \cite{CG83}, our proof of Theorem \ref{ILR2n} follows an approach similar to that in \cite{N26b} and is essentially based on the \textit{van Kampen obstruction} \cite{VK33}.

By Theorem \ref{ILR2n}, $S(\sigma_{2n+1}^{n-1})$ is not minimal with respect to this intrinsic linking property. On the other hand, for every positive integer $n$, $M^{(n)}$ is minimal in the following sense.

\begin{Proposition}\label{ILRn2min}
Let $N$ be a proper subcomplex of $M^{(n)}$. Then there exists an embedding $f \colon N \to \mathbb{R}^{2n}$ such that ${\rm lk}_2(f(\lambda)) = 0$ for every element $\lambda \in \Lambda^{n-1,n}(N)$.
\end{Proposition}

Note that $M^{(1)}$ is isomorphic to $N_{2}^{(1)} = K_1 \sqcup K_{2,3}$. On the other hand, the simplicial $n$-complexes $N_{1}^{(n)}$ and $N_3^{(n)}$ each have $2n+3$ vertices, while $N_2^{(n)}$ has $3n+3$ vertices \cite{N26b}. Since $M^{(n)}$ has $2n+4$ vertices, $M^{(n)}$ is not isomorphic to any of the simplicial $n$-complexes $N_{i}^{(n)}$ ($i=1,2,3$) for $n \ge 2$. We will also give a proof of Proposition \ref{ILRn2min} in the next section.

\begin{Remark}\label{k33}
Let $[3]^{*m}$ denote the $m$-fold join of three points. Note that $[3]^{*m}$ can naturally be regarded as a simplicial $(m-1)$-complex. For $n \ge 2$, it is also well known that $[3]^{*n}$ is a minimal simplicial complex that cannot be embedded in $\mathbb{R}^{2n-2}$ \cite{VK33}, \cite{Flores32}. In particular, $[3]^{*2}$ is isomorphic to $K_{3,3}$. In this case, $M_{[3]^{*m}}^{(n)}$ is isomorphic to $N_2^{(n)}$, and the same conclusion as in Theorem~\ref{ILR2n} holds. 
\end{Remark}

\section{The mod $2$ van Kampen obstruction and proof of Theorem \ref{ILR2n}}

Let $K$ be a simplicial $n$-complex. It is well known that every such $K$ admits a generic immersion into $\mathbb{R}^{2n}$. Here, an immersion $\varphi \colon K \to \mathbb{R}^{2n}$ is said to be \textit{generic} if all singularities of $\varphi(K)$ are transversal double points occurring between the interiors of pairs of disjoint $n$-simplices. For a generic immersion $\varphi$ of $K$ into $\mathbb{R}^{2n}$ and any two mutually disjoint $n$-simplices $s, s' \in \varDelta^n(K)$, let $l(\varphi(s),\varphi(s'))$ denote the number of double points between $\varphi(s)$ and $\varphi(s')$. Denote by $K^*$ the \emph{deleted product} of $K$, that is, the subcomplex of $K \times K$ consisting of all cells $s \times s'$ with $s \cap s' = \emptyset$. Let $\widetilde{K}^*$ be the quotient complex of $K^*$ obtained by identifying $s \times s'$ with $s' \times s$. Given a generic immersion $\varphi \colon K \to \mathbb{R}^{2n}$, the modulo-$2$ reduction of $l(\varphi(s),\varphi(s'))$ defines a $\mathbb{Z}_2$-cocycle $c_{\varphi}$ on $\widetilde{K}^*$, and its cohomology class $o_K = [c_{\varphi}] \in H^{2n}(\widetilde{K}^*;\mathbb{Z}_2)$ is called the \emph{mod $2$ van Kampen obstruction} of $K$.

Let $n$ be a positive integer, and let $a$, $b$, and $c$ be three vertices not belonging to $\sigma_{2n}^{n-1}$. In the following, we regard $\sigma_{2n}^{n-1} * \{a,b,c\}$ as a simplicial $n$-complex $K$.

\begin{Lemma}\label{cohomo}
Let $K = \sigma_{2n}^{n-1} * \{a,b,c\}$. Then 
\[
H^{2n}(\widetilde{K}^*;\mathbb{Z}_2) \cong \mathbb{Z}_2.
\]
Moreover, this group is generated by the cohomology class $[c]$ of the cocycle $c$ defined as follows. Fix any two disjoint $(n-1)$-simplices $\tau, \tau' \in \varDelta^{n-1}(\sigma_{2n}^{n-1})$, and let $\sigma = a*\tau$, $\sigma' = b*\tau'$. Define $c \in C^{2n}(\widetilde{K}^*;\mathbb{Z}_2)$ to be the dual cochain of the $2n$-cell $\sigma \times \sigma'$; that is,
\[
c(s \times s') = 
\begin{cases}
1 & \text{if } s \times s' = \sigma \times \sigma' \text{ in } \widetilde{K}^*, \\
0 & \text{otherwise}.
\end{cases}
\]
\end{Lemma}

\begin{proof}
For a pair of disjoint $n$-simplices $s,s'$ in $K$, let $E^{s,s'}$ denote the dual cochain of the $2n$-cell $s \times s'$ in $\widetilde{K}^*$. Similarly, for an $(n-1)$-simplex $t$ and an $n$-simplex $s'$ in $K$ with $t \cap s' = \emptyset$, let $V^{t,s'}$ denote the dual cochain of the $(2n-1)$-cell $t \times s'$ in $\widetilde{K}^*$.

First, fix an $(n-1)$-simplex $t = |a a_0 \cdots a_{n-2}|$ and an $n$-simplex $s' = |b a_{n+1} \cdots a_{2n}|$ in $K$. Note that $t$ and $s'$ are disjoint. Let 
\[
s_1 = |a a_0 \cdots a_{n-2} a_{n-1}|, \quad s_2 = |a a_0 \cdots a_{n-2} a_n|
\]
be two $n$-simplices containing $t$ as a face. Then the $(2n-1)$-coboundary satisfies
\[
\delta^{2n-1}(V^{t,s'}) = E^{s_1,s'} + E^{s_2,s'}.
\]
Hence $E^{s_1,s'}$ and $E^{s_2,s'}$ are cohomologous in $H^{2n}(\widetilde{K}^*;\mathbb{Z}_2)$. By the symmetry of $K$, this relation implies that for any pair of disjoint $n$-simplices $\sigma = a*\tau$ and $\sigma' = b*\tau'$ (where $\tau,\tau'\in\varDelta^{n-1}(\sigma_{2n}^{n-1})$ are disjoint), all such $E^{\sigma,\sigma'}$ are cohomologous to each other.

Next, fix an $(n-1)$-simplex $t' = |a_0 a_1 \cdots a_{n-1}|$ and an $n$-simplex $s' = |b a_{n+1} \cdots a_{2n}|$ in $K$. Let 
\[
s_1' = |a a_0 \cdots a_{n-1}|, \quad s_2' = |c a_0 \cdots a_{n-1}|
\]
be two $n$-simplices containing $t'$ as a face. Then, similarly,
\[
\delta^{2n-1}(V^{t',s'}) = E^{s_1',s'} + E^{s_2',s'}.
\]
By the symmetry of $K$, all $E^{\sigma,\sigma'}$ with $\sigma = a*\tau$ and $\sigma' = c*\tau'$ (where $\tau,\tau'\in\varDelta^{n-1}(\sigma_{2n}^{n-1})$ are disjoint) are also cohomologous to each other. The same type of relation holds for pairs of the form $\sigma = b*\tau$ and $\sigma' = c*\tau'$.

Combining these relations, we conclude that all dual cochains $E^{\sigma,\sigma'}$ corresponding to pairs of disjoint $n$-simplices in $K$ are cohomologous to each other. Therefore,
\[
H^{2n}(\widetilde{K}^*;\mathbb{Z}_2) \cong \mathbb{Z}_2,
\]
and this group is generated by the cohomology class $[c]$ of the dual cochain of any single $2n$-cell $\sigma \times \sigma'$, where $\sigma = a*\tau$ and $\sigma' = b*\tau'$ for disjoint $(n-1)$-simplices $\tau,\tau'\in\varDelta^{n-1}(\sigma_{2n}^{n-1})$. 
\end{proof}

\begin{Theorem}\label{vko}
For every generic immersion $\varphi$ of $K = \sigma_{2n}^{n-1}*\{a,b,c\}$ into $\mathbb{R}^{2n}$, the following holds:
\begin{eqnarray*}
\sum_{\substack{s,s'\in\varDelta^{n}(K)\\ s\cap s'=\varnothing}} l(\varphi(s),\varphi(s')) \equiv 1 \pmod{2}.
\end{eqnarray*}
\end{Theorem}

\begin{proof}
Let $J$ be a simplicial $(n-1)$-complex and let $a$, $b$, $c$ be three vertices not belonging to $J$. It is known that if the mod $2$ van Kampen obstruction $o_J$ of $J$ is non-zero, then the obstruction $o_{J*\{a,b,c\}}$ of the join $J*\{a,b,c\}$ is also non-zero \cite{Pa20}. Since the mod $2$ van Kampen obstruction $o_{\sigma_{2n}^{n-1}}$ is known to be non-zero, it follows that the obstruction $o_K$ of $K=\sigma_{2n}^{n-1}*\{a,b,c\}$ is also non-zero. 

On the other hand, by Lemma \ref{cohomo}, we have $H^{2n}(\widetilde{K}^*;\mathbb{Z}_2)\cong\mathbb{Z}_2$ and $o_K$ coincides with its unique non-trivial generator. Since all $2n$-cells of $\widetilde{K}^*$ are cohomologous, for every generic immersion $\varphi \colon K \to \mathbb{R}^{2n}$, we have
\begin{eqnarray*}
o_K = \bigg( \sum_{\substack{s,s'\in\varDelta^{n}(K)\\ s\cap s'=\emptyset}} l(\varphi(s),\varphi(s')) \bigg) \cdot [E^{\sigma,\sigma'}] \in H^{2n}(\widetilde{K}^*;\mathbb{Z}_2),
\end{eqnarray*}
where $\sigma = a*\tau$ and $\sigma' = b*\tau'$ for some disjoint $(n-1)$-simplices $\tau,\tau'$. Since $o_K$ is non-zero, the coefficient must be $1$ in $\mathbb{Z}_2$. This implies that the total number of double points is odd.
\end{proof}

\begin{proof}[Proof of Theorem \ref{ILR2n}]
For each $n$-simplex $\tau$ of $\sigma_{2n}$, the join $\partial\tau * \{a,b\}$ is homeomorphic to the $n$-sphere. Let $\Gamma^n$ be the collection of all such subcomplexes contained in $M^{(n)}$. Next, for each $(n-1)$-simplex $\tau'$ of $\sigma_{2n}^{n-1}$, we construct a subcomplex by taking $\tau'$ itself and attaching $\tau'' * c$ for every $(n-2)$-simplex $\tau''$ of $\tau'$ (when $n\ge 2$). For $n=1$, we take the subcomplex consisting of the two points $\{c, \tau'\}$. This subcomplex is homeomorphic to the $(n-1)$-sphere, as it coincides with the boundary of the $n$-simplex $\tau' * c$. Let $\Gamma^{n-1}$ be the collection of all such subcomplexes contained in $M^{(n)}$. Then every $\lambda \in \Lambda^{n-1,n}(M^{(n)})$ can be expressed as the disjoint union of an element $\gamma \in \Gamma^{n-1}$ and an element $\delta \in \Gamma^n$, see Fig.~\ref{M_2_link} for the case $n=2$.

Let $f \colon M^{(n)} \to \mathbb{R}^{2n}$ be an embedding. This can be extended to a generic immersion $\hat{f}$ of the complex $K = \sigma_{2n}^{n-1} * \{a,b,c\}$ into $\mathbb{R}^{2n}$ such that, for any two $n$-simplices $s,s' \in \varDelta^n(K)$ with $c \in s$ and $c \notin s'$, the images $\hat{f}(s)$ and $\hat{f}(s')$ intersect transversely at finitely many double points, all of which occur in the interiors of both simplices. 

Note that for any $\lambda = \gamma \sqcup \delta \in \Lambda^{n-1,n}(M^{(n)})$ with $\gamma = \partial s$, $s \in \varDelta^n(K)$ and $c \in s$, we have
\begin{eqnarray}\label{lkdef}
{\rm lk}_{2}(f(\gamma),f(\delta))
\equiv \sum_{s'\in \varDelta^{n}(\delta)} l(\hat{f}(s),\hat{f}(s')) \pmod{2}.
\end{eqnarray}
Note also that for any $\lambda = \partial s \sqcup \delta \in \Lambda^{n-1,n}(M^{(n)})$, each $n$-simplex $s' \in \varDelta^n(\delta)$ determines a unique unordered pair $\{s,s'\}$ with $s \cap s' = \emptyset$, and conversely, every such unordered pair $\{s,s'\} \subset \varDelta^n(K)$ with $c \in s$ and $c \notin s'$ arises from exactly one $\lambda$. 

Then, by \eqref{lkdef} and Theorem \ref{vko}, we obtain
\begin{eqnarray*}
\sum_{\lambda \in \Lambda^{n-1,n}(M^{(n)})} {\rm lk}_{2}(f(\lambda))
&=& \sum_{\substack{\gamma \sqcup \delta \in \Lambda^{n-1,n}(M^{(n)}) \\ \gamma \in \Gamma^{n-1},\ \delta \in \Gamma^n}} {\rm lk}_{2}(f(\gamma),f(\delta)) \\
&\equiv& \sum_{\substack{\gamma \sqcup \delta \in \Lambda^{n-1,n}(M^{(n)}) \\ \gamma = \partial s,\ s \in \varDelta^{n}(K),\ c \in s}}
\left( \sum_{s' \in \varDelta^{n}(\delta)} l(\hat{f}(s),\hat{f}(s')) \right) \\
&=& \sum_{\substack{s,\,s' \in \varDelta^{n}(K) \\ s \cap s' = \varnothing}} l(\hat{f}(s),\hat{f}(s')) \\
&\equiv& 1 \pmod{2}.
\end{eqnarray*}
Thus, there exists some $\lambda \in \Lambda^{n-1,n}(M^{(n)})$ such that ${\rm lk}_{2}(f(\lambda)) = 1$. 
\end{proof}

\begin{figure}[htbp]
\begin{center}
\scalebox{0.6}{\includegraphics*{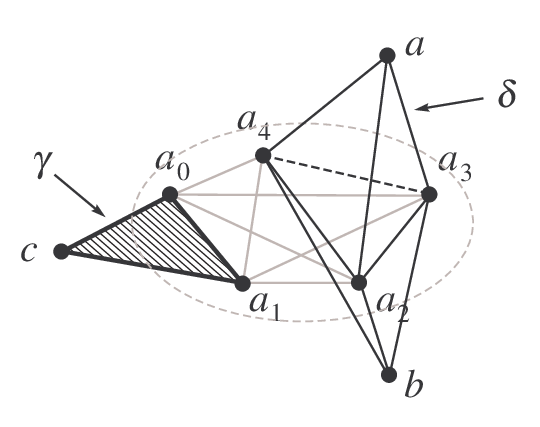}}
\caption{$\lambda=\gamma\sqcup \delta\in \Lambda^{1,2}(M^{(2)})$}
\label{M_2_link}
\end{center}
\end{figure}

\begin{proof}[Proof of Proposition \ref{ILRn2min}]
We can regard $M^{(n)}$ as follows. Let $\sigma_{2n+1}^{n-1}$ be the $(n-1)$-skeleton of the $(2n+1)$-simplex with vertices $a_0,a_1,\dots,a_{2n},a_{2n+1}$, where we identify the vertex $c$ with $a_{2n+1}$. Then $M^{(n)}$ is obtained from $\sigma_{2n+1}^{n-1}$ by attaching all $n$-simplices of the form $\rho * \{a,b\}$ with $\rho \in \varDelta^{n-1}(\sigma_{2n}^{n-1})$. It is known that $\sigma_{2n+1}^{n-1}$ has intrinsically linkedness with respect to its embedding into $\mathbb{R}^{2n-1}$, and in particular, there exists an embedding $h \colon \sigma_{2n+1}^{n-1} \to \mathbb{R}^{2n-1}$ whose image contains exactly one two-component link consisting of two $(n-1)$-spheres with ${\rm lk}_2 = 1$ \cite{T00}. Without loss of generality, by relabeling the vertices if necessary, we may assume that this two-component link is given by 
\begin{eqnarray*}
h(\partial |a_{n+1} \cdots a_{2n+1}|)\sqcup h(\partial |a_0 \cdots a_n|).
\end{eqnarray*}

Consider $\mathbb{R}^{2n-1}$ as a subspace of $\mathbb{R}^{2n}$. Place vertex $a$ in the half-space where the $2n$-th coordinate is positive and vertex $b$ where it is negative. For each $(n-1)$-simplex $\rho \in \varDelta^{n-1}(\sigma_{2n}^{n-1})$, take the join $h(\rho) * \{a,b\}$. This construction yields an embedding $f \colon M^{(n)} \to \mathbb{R}^{2n}$. In this embedding, for any $\lambda = \gamma \sqcup \gamma' \in \Lambda^{n-1,n-1}(\sigma_{2n+1}^{n-1})$ and $\mu = \gamma \sqcup (\gamma' * \{a,b\}) \in \Lambda^{n-1,n}(M^{(n)})$, the mod $2$ linking number of $f(\lambda)$ in $\mathbb{R}^{2n-1}$ coincides with that of $f(\mu)$ in $\mathbb{R}^{2n}$. Hence the two-component link
\begin{eqnarray*}
f(\partial |a_{n+1} \cdots a_{2n+1}|) \ \sqcup \ f(\partial |a_0 \cdots a_n| * \{a,b\}),
\end{eqnarray*}
consisting of an $(n-1)$-sphere and an $n$-sphere, is the unique two-component llink in $f(M^{(n)})$ with ${\rm lk}_2 = 1$.

Now define two subcomplexes $N_1$ and $N_2$ of $M^{(n)}$ by
\begin{eqnarray*}
N_1 &=& M^{(n)} \setminus \bigl\{ |a_{n+2} \cdots a_{2n}| * a_{2n+1} \bigr\}, \\
N_2 &=& M^{(n)} \setminus \bigl\{ |a_0 \cdots a_{n-1}| * a \bigr\}.
\end{eqnarray*}
Since $N_1$ does not contain the $(n-1)$-simplex $|a_{n+2} \cdots a_{2n}| * a_{2n+1}$, the first $n$-sphere $f(\partial |a_0 \cdots a_n| * \{a,b\})$ no longer appears in $f(N_1)$. Similarly, $f(N_2)$ does not contain the second sphere. Therefore, neither $f(N_1)$ nor $f(N_2)$ contains any two-component link consisting of an $(n-1)$-sphere and an $n$-sphere with ${\rm lk}_2 = 1$.

By the symmetry of $M^{(n)}$, every maximal proper subcomplex of $M^{(n)}$ (i.e., every subcomplex obtained by deleting exactly one $n$-simplex) is isomorphic to either $N_1$ or $N_2$. This implies that for any proper subcomplex $N$ of $M^{(n)}$, there exists an embedding $f \colon N \to \mathbb{R}^{2n}$ such that ${\rm lk}_2(f(\lambda)) = 0$ for every $\lambda \in \Lambda^{n-1,n}(N)$.
\end{proof}


%
{\normalsize
}

\end{document}